\documentclass[12pt]{amsart}
\usepackage{amsfonts,graphicx,amssymb,amscd,amsmath,enumerate,verbatim,calc}

\def\NZQ{\mathbb}               

\def\QQ{{\NZQ Q}}
\def\ZZ{{\NZQ Z}}
\def\RR{{\NZQ R}}
\def\opn#1#2{\def#1{\operatorname{#2}}} 
\opn\cone{cone}
\opn\lex{lex}
\opn\rev{rev}
\opn\Syz{Syz} \opn\Im{Im} \opn\Ker{Ker} \opn\Coker{Coker}
\opn\Hom{Hom} \opn\Tor{Tor} \opn\Ext{Ext}
\opn\End{End} \opn\Aut{Aut} \opn\id{id} \opn\nat{nat}
\opn\mod{mod} \opn\ord{ord}
\opn\aff{aff} \opn\con{conv} \opn\relint{relint} \opn\st{st}
\opn\lk{lk} \opn\cn{cn} \opn\core{core} \opn\vol{vol}
\opn\link{link} \opn\star{star} \opn\sgn{sgn}
\def\Ac{{\mathcal A}}

\def\Fc{{\mathcal F}}
\def\Pc{{\mathcal P}}
\def\Qc{{\mathcal Q}}
\theoremstyle{thmit} 
\newtheorem{Theorem}{Theorem}[section]
\newtheorem{Lemma}[Theorem]{Lemma}
\newtheorem{Corollary}[Theorem]{Corollary}
\newtheorem{Proposition}[Theorem]{Proposition}

\newtheorem{Conjecture}[Theorem]{Conjecture}

\newtheorem{Remark}[Theorem]{Remark}

\newtheorem{Example}[Theorem]{Example}

\newcommand{\set}[1]{\left\{#1\right\}}  
\newcommand{\with}{\ \vrule\ }  
\newcommand{\defa}{:=}
\newcommand{\bvec}[1]{b_{#1}}
\title{Normal cyclic polytopes and cyclic polytopes that are not very ample}
\author{Takayuki Hibi}
\address{Department of Pure and Applied Mathematics, \\
Graduate School of Information Science and Technology, \\
Osaka University, Toyonaka, Osaka 560-0043, Japan \\
hibi@math.sci.osaka-u.ac.jp}
\author{Akihiro Higashitani}
\address{Department of Pure and Applied Mathematics, \\
Graduate School of Information Science and Technology, \\
Osaka University, Toyonaka, Osaka 560-0043, Japan \\
a-higashitani@cr.math.sci.osaka-u.ac.jp}
\author{Lukas Katth\"an}
\address{Fachbereich Mathematik und Informatik, \\
Philipps-Universit\"at Marburg, 35032 Marburg, Germany \\
katthaen@mathematik.uni-marburg.de}%
\author{Ryota Okazaki}
\address{Faculty of Education, Fukuoka University of Education, \\
Munakata, Fukuoka 811-4192, Japan \\
rokazaki@fukuoka-edu.ac.jp}
\thanks{
{\bf Acknowledgements:}
The first and forth authors are supported by the JST CREST 
``Harmony of Gr\"obner Bases and the Modern Industrial Society.'' \\
The second author is supported by JSPS Research Fellowship for Young Scientists. \\
This research was performed while the third author was staying
at Department of Pure and Applied Mathematics, 
Osaka University, November 2011 -- April 2012, supported by the DAAD
}

\begin{document}

\maketitle

\begin{abstract}
Let $d$ and $n$ be positive integers
with $n \geq d + 1$ and $\tau_{1}, \ldots, \tau_{n}$ integers 
with $\tau_{1} < \cdots < \tau_{n}$.  
Let $C_{d}(\tau_{1}, \ldots, \tau_{n}) \subset \RR^{d}$ 
denote the cyclic polytope of dimension $d$ with $n$ vertices
$(\tau_{1},\tau_{1}^{2},\ldots,\tau_{1}^{d}), \ldots,  
(\tau_{n},\tau_{n}^{2},\ldots,\tau_{n}^{d})$.
We are interested in finding the smallest integer $\gamma_{d}$
such that if $\tau_{i+1} - \tau_{i} \geq \gamma_{d}$ for $1 \leq i < n$, 
then $C_{d}(\tau_{1}, \ldots, \tau_{n})$ is normal.
One of the known results is $\gamma_{d} \leq d (d + 1)$.  
In the present paper a new inequality $\gamma_{d} \leq d^{2} - 1$ is proved.
Moreover, it is shown that  
if $d \geq 4$ with $\tau_{3} - \tau_{2} = 1$, then
$C_{d}(\tau_{1}, \ldots, \tau_{n})$ 
is not very ample. 
\end{abstract}

\section*{Introduction}
The cyclic polytope is one of the most distinguished polytopes and 
played the essential role in the classical theory of convex polytopes (\cite{Gru}).  
Let $d$ and $n$ be positive integers
with $n \geq d + 1$ and $\tau_{1}, \ldots, \tau_{n}$ real numbers 
with $\tau_{1} < \cdots < \tau_{n}$.  
The convex polytope $C_{d}(\tau_{1}, \ldots, \tau_{n})$
which is the convex hull of the finite set
\[
\{
(\tau_{1},\tau_{1}^{2},\ldots,\tau_{1}^{d}), \ldots, (\tau_{n},\tau_{n}^{2},\ldots,\tau_{n}^{d}) 
\} \subset \RR^{d}
\]
is called a {\em cyclic polytope}.
It is known that $C_{d}(\tau_{1}, \ldots, \tau_{n})$
is a simplicial polytope of dimension $d$ with $n$ vertices.
The combinatorial type of $C_{d}(\tau_{1}, \ldots, \tau_{n})$ 
is independent of the particular choice of   
real numbers $\tau_{1}, \ldots, \tau_{n}$.

The present paper is devoted to the study on {\em integral} cyclic polytopes. 
A convex polytope is called integral if all of its vertices have integer coordinates.
The integral convex polytope has established an active area lying between  
combinatorics and commutative algebra (\cite{HibiRedBook, StanleyGreenBook}).

Let, in general, $\Pc \subset \RR^{N}$ be an integral convex polytope, 
define $\Pc^{*} \subset \RR^{N+1}$ to be the convex hull of all points
$(1, \alpha) \in \RR^{N+1}$ with $\alpha \in \Pc$ and let $\Ac_\Pc = \Pc^{*} \cap \ZZ^{N+1}$
denote the set of integer points in $\Pc^{*}$.
Let $\ZZ_{\geq 0}$ denote the set of nonnegative integers and
$\QQ_{\geq 0}$ the set of nonnegative rational numbers.

We say that $\Pc$ is {\em normal}
if one has 
\[
\ZZ_{\geq 0}\Ac_\Pc = \ZZ\Ac_\Pc \cap \QQ_{\geq 0}\Ac_\Pc.
\]
Moreover, $\Pc$ is called {\em very ample}
if the set
\[
(\ZZ\Ac_\Pc \cap \QQ_{\geq 0}\Ac_\Pc) \setminus \ZZ_{\geq 0}\Ac_\Pc
\]
is finite. 
One of the most fundamental questions on integral convex polytopes is
to determine whether a given integral convex polytope is normal (\cite{OHnormal}).

On the other hand, we say that an integral convex polytope
$\Pc \subset \RR^{N}$ has the {\em integer decomposition property}
if, for each $m = 1, 2, \ldots$ and for each $\alpha \in m \Pc \cap \ZZ^{N}$,
there exist $\alpha_{1}, \ldots, \alpha_{m}$ belonging to $\Pc \cap \ZZ^{N}$
such that $\alpha = \alpha_{1} + \cdots + \alpha_{m}$.
Here $m \Pc = \{ \, m \alpha \, : \, \alpha \in \Pc \, \}$.
If $\Pc$ has the integer decomposition property, then $\Pc$ is normal.
However, the converse is false.  For example, the tetrahedron 
${\mathcal T}_{3} \subset \RR^{3}$
with the vertices $(0,0,0), (1,1,0), (1,0,1)$ and $(0,1,1)$ is normal, but
cannot have the integer decomposition property because
$(1,1,1) \in 2 {\mathcal T}_{3}$.
If $\Pc \subset \RR^{d}$ is an integral convex polytope of dimension $d$ with 
$\ZZ (\Pc^{*} \cap \ZZ^{d+1}) = \ZZ^{d+1}$, then 
$\Pc$ has the integer decomposition property if and only if $\Pc$ is normal. 
Lemma \ref{lemma:basis} says that 
every integral cyclic polytope $\Pc \subset \RR^d$ 
satisfies $\ZZ (\Pc^* \cap \ZZ^{d+1}) = \ZZ^{d+1}$. 
In particular it follows that an integral cyclic polytope is normal 
if and only if it has the integer decomposition property.

Let, as before, $d$ and $n$ be positive integers with $n \geq d + 1$.
Given integers $\tau_{1}, \ldots, \tau_{n}$ 
with $\tau_{1} < \cdots < \tau_{n}$, we wish to examine whether 
$C_{d}(\tau_{1}, \ldots, \tau_{n})$ is normal or not.
Thus our final goal is to classify the integers 
$\tau_{1}, \ldots, \tau_{n}$ 
with $\tau_{1} < \cdots < \tau_{n}$ for which 
$C_{d}(\tau_{1}, \ldots, \tau_{n})$ is normal.
Even though to find a complete classification seems to be rather difficult,
many fascinating problems arise in the natural way.
As a first step toward our goal,
we are interested in finding the smallest integer 
$\gamma_{d}$ such that if $\tau_{i+1} - \tau_{i} \geq \gamma_{d}$ 
for $1 \leq i < n$, 
then $C_{d}(\tau_{1}, \ldots, \tau_{n})$ is normal.
Since the lattice length of each edge $\con(\{(\tau_i,\ldots,\tau_i^d),(\tau_j,\ldots,\tau_j^d)\})$ 
of $C_{d}(\tau_{1}, \ldots, \tau_{n})$ 
coincides with $|\tau_j-\tau_i|$, it follows immediately from 
\cite[Theorem 1.3 (b)]{Gubel}
that one has $\gamma_{d} \leq d (d + 1)$.  
In the present paper a new inequality $\gamma_{d} \leq d^{2} - 1$ is proved
(Theorem \ref{main1}).
Moreover, it is shown that  
if $d \geq 4$ with $\tau_{3} - \tau_{2} = 1$, then
$C_{d}(\tau_{1}, \ldots, \tau_{n})$ 
is not very ample (Theorem \ref{main2}).

A brief overview of the present paper is as follows. 
After preparing notation, terminologies together with several lemmata 
in Section $1$, a proof of Theorem \ref{main1}
is achieved in Section $2$.
Moreover, Section $3$ is devoted to showing 
Theorem \ref{main2}. 

Finally the study on an algebraic aspect of integral cyclic polytopes 
including toric rings of integral cyclic polytopes will be done 
in the forthcoming paper \cite{HHKO}.


\section{Preliminaries}

In this section, we prepare notation and lemmata for our main theorem. 

\smallskip

First of all, we will review some fundamental facts on cyclic polytopes.
Let $d$ and $n$ be positive integers with $n \geq d+1$. 
It is convenient to work with a homogeneous version of the cyclic polytopes, hence, 
throughout the present paper, we consider 
$C^{*}_d( \tau_1, \ldots, \tau_n)$ instead of $C_d( \tau_1, \ldots, \tau_n)$. 
For $n$ real numbers $\tau_1, \ldots, \tau_n$ with $\tau_1 < \cdots < \tau_n$, we set 
\[ 
v_i \defa (1, \tau_i, \tau_i^2, \ldots, \tau_i^d) \in \RR^{d+1} \; \text{ for } \; 1 \leq i \leq n. 
\]
In other words, $C^{*}_d( \tau_1, \ldots, \tau_n) = \con(\{ v_i : 1 \leq i \leq n \}) \subset \RR^{d+1}$. 
Unless stated otherwise, 
we will always assume the indices are ordered like $\tau_1 < \ldots < \tau_n$. 
See \cite[Chapter 0]{ziegler} for some basic properties of cyclic polytopes. 
We will use a well-known characterization of their facets. (See, e.g., \cite[Theorem 0.7]{ziegler}). 

Let $[n] \defa \set{1,\ldots,n}$ and 
let us say that a set $S \subset [n]$ {\em forms a facet} of $C^{*}_d( \tau_1, \ldots, \tau_n)$ 
if $\con(\set{v_i : i \in S})$ is its facet. 
\begin{Proposition}[Gale's evenness condition] \label{galesevenness}
A set $S \subset [n]$ with $d$ elements forms a facet of $C^{*}_d( \tau_1, \ldots, \tau_n)$ 
if and only if $S$ satisfies the following condition: 
If $i$ and $j$ with $i < j$ are not in $S$, then the number of elements of $S$ between $i$ and $j$ is even. 
In other words, 
\[ 2 \; \vrule\: \# \set{k \in S \with i < k < j}, \] 
where $\# X$ stands for the number of elements contained in a finite set $X$. 
\end{Proposition}

Hereafter, we will assume that $\tau_1,\ldots,\tau_n$ are integers. 

Let $\Delta_{ij} \defa \tau_j - \tau_i$ for $i,j \in [n]$. 
The proof of Proposition \ref{galesevenness} yields a description of 
the inequality of the supporting hyperplane defining each facet. 
Let $S=\set{k_1,\ldots,k_d} \subset [n]$ and consider the polynomial 
\[
\sum_{i=0}^d c_{S,i} t^i \defa \prod_{i \in S} (t - \tau_i) \,.
\]
Then all $d$ vectors $v_{k_1},\ldots,v_{k_d}$ vanish by the linear form 
\[ 
\sigma_S: \RR^{d+1} \ni (w_0,w_1,\dotsc,w_d) \mapsto \sum_{i=0}^d c_{S,i} w_i \in \RR \,,
\] 
thus it defines the hyperplane spanned by them. Note that we index the first coordinate by $0$.
Hence, if the set $S$ forms a facet $\Fc$ of $\Pc^*=C^{*}_d( \tau_1, \ldots, \tau_n)$, 
then $\sigma_S$ is the linear form defining $\Fc$, 
which means that $\sigma_S(x) \geq 0$ if $x$ is in $\Pc^*$ and $\sigma_S(x)=0$ if $x$ is in $\Fc$.  
For every $j \in [n] \setminus S$, it holds $\sigma_S(v_j) = \prod_{i \in S} \Delta_{ij}$. 
This has a useful implication, that is, if we write a vector $x \in \ZZ^{d+1}$ as 
$x = \sum_{i \in S} \lambda_i v_i + \lambda_j v_j$ with rational coefficients $\lambda_i$, 
then the denominator of $\lambda_j$ is a divisor of $\prod_{i \in S} \Delta_{ij}$, 
because $\sigma_S(x) = \lambda_j \prod_{i \in S} \Delta_{ij}$ is an integer.

We introduce a special representation of cyclic polytopes which is sometimes helpful. 
Write the vectors $v_1,\ldots,v_n$ as row vectors into a matrix, namely, 
\begin{equation}\label{eq:matrix}
\begin{pmatrix}
v_1 \\
v_2 \\
\vdots \\
v_n
\end{pmatrix}=
\begin{pmatrix}
1      &\tau_1&\tau_1^2& \dots &\tau_1^d \\
1      &\tau_2&\tau_2^2& \dots &\tau_2^d \\
\vdots &\vdots&        &       &\vdots   \\
1      &\tau_n&\tau_n^2& \dots &\tau_n^d \\
\end{pmatrix}. 
\end{equation}

\begin{Lemma}\label{lemma:delta}
The aforementioned matrix can be transformed to the following matrix 
by using a unimodular transformation: 
\begin{equation}\label{eg:matrixdelta}
\begin{pmatrix}
1& 0             & \cdots                        &\cdots & 0 \\
1&\Delta_{12}    &  0                            & \ddots& \vdots \\
1&\Delta_{13}    &  \Delta_{13}\Delta_{23}       & \ddots& \vdots \\
\vdots&\vdots    &               \vdots          & \ddots& 0 \\
1&\Delta_{1, d+1}&\Delta_{1, d+1}\Delta_{2, d+1} & \dots & \prod_{k=1}^d\Delta_{k, d+1} \\
\vdots&\vdots    &               \vdots          &       & \vdots \\
1&\Delta_{1, n}  &\Delta_{1, n}\Delta_{2, n}     & \dots & \prod_{k=1}^d\Delta_{k, n} \\
\end{pmatrix}.
\end{equation}
In particular, the convex hull of the row vectors of this matrix is unimodularly equivalent to 
$C^{*}_d( \tau_1, \ldots, \tau_n)$. 
\end{Lemma}
A proof of the above lemma is essentially the same as a proof of the well-known Vandermonde determinant identity. 
Note that Lemma \ref{lemma:delta} is valid for any ordering of the parameters $\tau_1,\dotsc,\tau_n$, 
i.e., any ordering of $v_1,\ldots,v_n$.

Let us identify a special case where the polytopes are indeed unimodularly equivalent. 
\begin{Lemma}\label{equiv}
An integral cyclic polytope $C^{*}_d(\tau_1,\ldots,\tau_d)$ is unimodularly equivalent to 
$C^{*}_d(-\tau_n, \ldots, -\tau_1)$. Moreover, for any integer $m$, 
$C^{*}_d(\tau_1,\ldots,\tau_d)$ is unimodularly equivalent to $C^{*}_d(\tau_1 + m , \ldots, \tau_n + m)$. 
\end{Lemma}
\begin{proof}
The replacement $\tau_i \mapsto - \tau_i$ corresponds to a multiplication with $-1$ 
in every column of \eqref{eq:matrix} with an odd exponent. This is a unimodular transformation.
The second statement is immediate from Lemma \ref{lemma:delta}, 
because the matrix \eqref{eg:matrixdelta} depends only on the differences $\Delta_{ij} = \tau_j - \tau_i$.
\end{proof}

We define a certain class of vectors which we will use in the sequel. 
Let $S=\{i_1,\ldots,i_q\} \subset [n]$ be a non-empty set, where $i_1< \cdots <i_q$. 
Then we define 
\[ \bvec{S} \defa \sum_{i\in S} \frac{1}{\prod_{j \in S \setminus \{i\}} \Delta_{ij}} v_i 
=\sum_{k=1}^q \frac{(-1)^{k+1}}{\prod_{j \in S \setminus \{i_k\}} |\Delta_{i_kj}|}v_{i_k}, \]
where $\bvec{S}=v_{i_1}$ when $q=1$, i.e., $\# S=1$. 
If $S$ is small, we will sometimes omit the brackets around the elements, 
thus we write, for example, $\bvec{ij} = \bvec{\set{i,j}}$. 
However, the vector does not depend on the order of the indices.
\begin{Example}{\em 
Let us write down $\bvec{S}$'s for small sets $S$. Assume $1\leq i<j<k<l\leq n$. Then 
\begin{align*}
\bvec{i} &= v_i, \\
\bvec{ij} &= \frac{1}{\Delta_{ij}} v_i - \frac{1}{\Delta_{ij}} v_j, \\
\bvec{ijk} &= \frac{1}{\Delta_{ij}\Delta_{ik}} v_i - \frac{1}{\Delta_{ij}\Delta_{jk}} v_j 
+ \frac{1}{\Delta_{ik}\Delta_{jk}} v_k, \\
\bvec{ijkl} &= \frac{1}{\Delta_{ij}\Delta_{ik}\Delta_{il}} v_i - \frac{1}{\Delta_{ij}\Delta_{jk}\Delta_{jl}} v_j 
+ \frac{1}{\Delta_{ik}\Delta_{jk}\Delta_{kl}} v_k - \frac{1}{\Delta_{il}\Delta_{jl}\Delta_{kl}} v_l. 
\end{align*}
The sign changes are due to a reordering of the indices since $\Delta_{ij} = - \Delta_{ji}$. 
If $v_i,v_j,v_k,v_l$ are given in the form \eqref{eg:matrixdelta}, i.e., if 
\begin{equation*}
\begin{pmatrix}
v_i \\
v_j \\
v_k \\
v_l
\end{pmatrix}=
\begin{pmatrix}
1&0           &\cdots                 &\cdots                            &\cdots &\cdots &0 \\
1&\Delta_{ij} &0                      &\ddots                            &\cdots &\cdots &\vdots \\
1&\Delta_{ik} &\Delta_{ik}\Delta_{jk} &\ddots                            &\cdots &\cdots &\vdots \\
1&\Delta_{il} &\Delta_{il}\Delta_{jl} &\Delta_{il}\Delta_{jl}\Delta_{kl} &0      &\cdots &0\\
\end{pmatrix}, 
\end{equation*}
then $\bvec{i}=(1,0,\ldots,0)$, $\bvec{ij}=(0,-1,0,\ldots,0)$, $\bvec{ijk}=(0,0,1,0,\ldots,0)$ 
and $\bvec{ijkl}=(0,0,0,-1,0,\ldots,0)$. 
In general, $\bvec{1}, \bvec{12}, \ldots, \bvec{12 \cdots d+1}$ look like 
$(0,\ldots,0,\pm 1,0,\ldots,0)$ when $v_1,\ldots,v_{d+1}$ are of the form \eqref{eg:matrixdelta}. 
}\end{Example} 
\noindent The following proposition collects the basic properties on these vectors. 
\begin{Proposition}\label{prop:bvectors}
\begin{enumerate}
\item For any non-empty set $S \subset [n]$, one has $\bvec{S} \in \ZZ^{d+1}$. 
\item Let $S \subset [n]$ and $a,b\in S$ with $a\neq b$. Then we have a recursion formula 
\[ \bvec{S} = \frac{1}{\Delta_{ba}} \bvec{S\setminus \{a\}} + \frac{1}{\Delta_{ab}} \bvec{S\setminus \{b\}}. \]
\item For any distinct $d+1$ indices $i_1,\dotsc,i_{d+1} \in [n]$ (not necessarily ordered), the vectors
\[ \bvec{i_1}, \bvec{i_1 i_2}, \bvec{i_1 i_2 i_3}, \ldots, \bvec{i_1 \cdots i_{d+1}} \]
form a $\ZZ$-basis for $\ZZ^{d+1}$. 
\item If $\# S \geq d+2$, then $\bvec{S}=0$.
\end{enumerate}
\end{Proposition}
\begin{proof}
The second statement can be verified by elementary computations, 
using $\Delta_{ij} + \Delta_{jk} = \Delta_{ik}$ for $i,j,k\in [n]$.
To prove the first statement, we consider the components of $b_S$ 
as rational functions in $\tau_i, i\in S$. By induction on $\# S$, 
we prove the following statement. The components of $b_S$ are symmetric polynomials in $\tau_i, i\in S$, 
and their coefficients depend only on $\# S$.

If $\# S=1$, then $\bvec{S}=\bvec{i} = v_i = (1,\tau_i, \tau_i^2, \dotsc, \tau_i^d)$, 
thus the claim holds. Now consider a set $S$ with at least two distinct elements $a,b$. Let 
\[ f_j(\tau_a, \tau_i, i\in S), \; f_j(\tau_b, \tau_i, i\in S) \] 
be the $j$-th components of $\bvec{S\setminus b}$, $\bvec{S\setminus a}$, respectively. 
Then the difference between these polynomials is zero if we set $\tau_a = \tau_b$, hence the quotient
\[ \frac{f_j(\tau_a, \tau_i, i\in S) - f_j(\tau_b, \tau_i, i\in S)}{\tau_a - \tau_b} \]
is a polynomial as claimed. It is obviously symmetric in $a$ and $b$. 
Since we are free to choose any two elements of $S$, it is symmetric in all variables. 
The coefficients of the polynomial depend only on $\# S$, so the claim is proven. 
Note that the degree of the polynomial decreases by one by taking the quotient. 
Since the degree of the components of $v_i$ is at most $d+1$, we conclude that $b_S = 0$ for $\#S \geq d+2$.

To prove the third statement, we first note that 
the vertices $v_{i_1}, \dotsc, v_{i_{d+1}}$ are linearly independent. 
Take an element $x \in \ZZ^{d+1}$ and write it as $x = \sum \lambda_j v_{i_j}$. 
By considering $\sigma_{\set{i_1,\ldots,i_d}}(x)$, we can say that 
the coefficient $\lambda_{i_{d+1}}$ is of the form 
\[\lambda_{i_{d+1}} = \frac{k}{\prod_{j=1}^d \Delta_{i_j i_{d+1}}} \]
for an integer $k$. Thus, $x +(-1)^d k \bvec{i_1\dotsc i_{d+1}} \in \ZZ^{d+1}$ 
is a vector in the subspace spanned by $v_{i_1},\dotsc,v_{i_d}$. 
These vectors define a $(d-1)$-dimensional cyclic polytope again, 
so we can proceed by induction and obtain a representation of $x$ 
as a $\ZZ$-linear combination of the $\bvec{i_1},\bvec{i_1i_2},\dotsc,\bvec{i_1\dotsc i_{d+1}}$.
\end{proof}
\noindent We apply this construction to prove another useful fact on cyclic polytopes. 
\begin{Lemma}\label{lemma:basis}
For an integral cyclic polytope $\Pc \subset \RR^d$ of dimension $d$, 
one has 
\[\ZZ \Ac_{\Pc}=\ZZ^{d+1}\,.\] 
\end{Lemma}
\begin{proof}
First, we notice that $\ZZ \Ac_{\Pc} \subset \ZZ^{d+1}$ is obvious. 
To prove another inclusion, we construct a basis of $\ZZ^{d+1}$ from $d+1$ points in $\Ac_\Pc$.
We choose $d+1$ vertices $v_1,\dotsc,v_{d+1}$ of $\Pc^*$ and consider the vectors 
\[ \bvec{i_{d+1}}, \bvec{i_{d+1}} + \bvec{i_{d}i_{d+1}},\bvec{i_{d+1}} + \bvec{i_{d}i_{d+1}} + \bvec{i_{d-1}i_{d}i_{d+1}},\dotsc,
\sum_{l=1}^{d+1} \bvec{i_l\dotsc i_{d+1}}.\] 
Let us denote them by $c_j \defa \sum_{l=j}^{d+1} b_{i_l\dotsc i_{d+1}}$ for $j=1,\ldots,d+1$. 
By Proposition \ref{prop:bvectors} (3), they constitute a $\ZZ$-basis of $\ZZ^{d+1}$. 
Hence, if each $c_j$ is contained in $\Pc^{*}$, then our claim follows. 
For this, let us consider the coefficient of a vertex $v_{i_k}$ in the sequence of 
\[ \bvec{i_{d}},\bvec{i_{d}i_{d+1}}, \bvec{i_{d-1}i_{d}i_{d+1}}, \dotsc, \bvec{i_1\dotsc i_{d+1}}.\]
The coefficient of $v_{i_k}$ appears first in $\bvec{i_k\dotsc i_{d+1}}$, where it has a positive sign. 
After that, its sign is alternating and the absolute value is non-increasing since the denominators increase. 
Hence, the sum of those coefficients and thus the coefficient in $c_j$ is nonnegative. 
So, $c_j$ is a convex combination of the vertices of $\Pc^*$.
\end{proof}

Finally, we discuss the normality of integral cyclic polytopes. 
\begin{Lemma}\label{covering}
Let $\Pc$ be an integral cyclic polytope of dimension $d$. 
If any simplex of dimension $d$ whose vertices are chosen from those of $\Pc$ is normal, 
then $\Pc$ itself is also normal. 
\end{Lemma}
\begin{proof}
Let $v_1,\ldots,v_n$ be the vertices of $\Pc^*$. 
A proof is a direct application of Carath\'{e}odory's Theorem (see, e.g., \cite[Section 7]{Sch}). 
Let $x \in \ZZ\Ac_\Pc \cap \QQ_{\geq 0}\Ac_\Pc$. 
Now, Carath\'{e}odory's Theorem guarantees that there exist $d+1$ vertices $v_{i_1},\ldots,v_{i_{d+1}}$ of $\Pc^*$ 
such that $x \in \ZZ\Ac_\Qc \cap \QQ_{\geq 0}\Ac_\Qc$, where $\Qc=\con(\set{v_{i_1},\ldots,v_{i_{d+1}}})$. 
Here we use that $\ZZ\Ac_\Pc=\ZZ^{d+1} = \ZZ\Ac_\Qc$ by Lemma \ref{lemma:basis}.
If $\Qc$ is normal, then we have $x \in \ZZ_{\geq 0}\Ac_\Qc$, 
in particular, $x \in \ZZ_{\geq 0}\Ac_\Pc$. This implies that $\Pc$ is normal. 
\end{proof}


\section{Normal cyclic polytopes}

Our goal of this section is to prove 
\begin{Theorem}\label{main1}
Work with the same notations as in Section 1. 
If $\Delta_{i,i+1} \geq d^2-1$ for $1 \leq i  \leq n-1$, 
then $\Pc=C_d(\tau_1,\ldots,\tau_n)$ is normal. 
In particular, $\gamma_d \leq d^2-1$. 
\end{Theorem}

Most parts of this section are devoted to proving the simplex case. 
In fact, once we know that $\Pc$ is always normal 
when $n=d+1$ and $\Delta_{i,i+1} \geq d^2-1$ for $1 \leq i \leq d$, 
Theorem \ref{main1} follows immediately from Lemma \ref{covering}.

\bigskip 

Before giving a proof, we prepare two lemmata, Lemma \ref{key} and Lemma \ref{Z}. 
First, for Lemma \ref{key}, we start from proving 
\begin{Proposition}\label{minmax}
Let $(r_1,r_2,\ldots,r_{d+1}) \in \QQ^{d+1}$ satisfying 
$$0 \leq r_1 \leq r_2 \leq \cdots \leq r_{d+1} \leq 1 \;\; \text{ and } \;\; 
\sum_{i=1}^{d+1} r_i =m.$$ Then one has 
\begin{center}
{\em (a)} $\sum_{i=1}^jr_i \leq \frac{jm}{d+1}$ \;\; and \;\; 
{\em (b)} $\sum_{i=1}^jr_{d+2-i} \geq \frac{jm}{d+1}$ 
\end{center}
for any integer $j$ with $1 \leq j \leq d+1$. 
\end{Proposition}
\begin{proof}
We prove by induction on $j$. 

First, we show $r_1 \leq \frac{m}{d+1}$. 
Suppose that $r_1 > \frac{m}{d+1}$. 
Then one has $r_i>\frac{m}{d+1}$ for all $1 \leq i \leq d+1$ by $r_1 \leq r_2 \leq \cdots \leq r_{d+1}$. 
Thus, $m=\sum_{i=1}^{d+1}r_i>(d+1) \cdot \frac{m}{d+1}=m,$ a contradiction. 
Similarly, we also have $r_{d+1} \geq \frac{m}{d+1}$. 

Now, we assume that the assertions (a) and (b) hold for any integer $j'$ with $1 \leq j' < j$, 
where $j$ is some integer with $2 \leq j \leq d+1$. 
Let $d+1=kj+q$, where $k$ is a positive integer and $0 \leq q \leq j-1$, 
i.e., $k$ (resp. $q$) is a quotient (resp. a remainder) of $d+1$ divided by $j$. 
Suppose that $\sum_{i=1}^jr_i > \frac{jm}{d+1}$. Then one has 
$$
\sum_{i=1}^jr_{(k-1)j+i} \geq \sum_{i=1}^jr_{(k-2)j+i} \geq \cdots \geq 
\sum_{i=1}^jr_i > \frac{jm}{d+1}. 
$$
Moreover, by the hypothesis of induction, one also has 
$\sum_{i=kj+1}^{d+1}r_i=\sum_{i=1}^qr_{d+2-i} \geq \frac{mq}{d+1}$ 
when $q \not=0$. Hence, we obtain 
$$m=\sum_{i=1}^{d+1}r_i > k\cdot\frac{jm}{d+1}+\frac{mq}{d+1}=m\cdot\frac{kj+q}{d+1}=m,$$ 
a contradiction. Therefore, the assertion (a) also holds for $j$. 
Similarly, we also have the assertion (b) for $j$, as required. 
\end{proof}

\begin{Lemma}\label{key}
Let $d$ be a positive integer and $(r_1,r_2,\ldots,r_{d+1}) \in \QQ^{d+1}$ satisfying that 
$0 \leq r_1 \leq r_2 \leq \cdots \leq r_{d+1} \leq 1$ and that 
$\sum_{i=1}^{d+1} r_i$ is an integer which is greater than 1. Then one has 
\begin{eqnarray}\label{eq}
\max_{\substack{1 \leq i_1 < i_2 < \cdots < i_l \leq d+1, \\ 2 \leq l \leq d}}
\left\{ \sum_{j=1}^lr_{i_j} : \sum_{j=1}^{l-1}r_{i_j} \leq 1 \right\} \geq 1+\frac{1}{d+1}. 
\end{eqnarray}
\end{Lemma}
\begin{proof}
Let $m=\sum_{i=1}^{d+1}r_i$. When $m > d$, it must be satisfied that $r_i = 1$ 
for $1 \leq i \leq d+1$ and $m=d+1$ by our assumption. 
Thus, we may assume that $2 \leq m \leq d$. 
Let $M$ denote the value of the left-hand side of \eqref{eq}. 

{\bf The first step.} 
Assume that $m-1 > \lfloor \frac{d+1}{2} \rfloor$. 
Then, by Proposition \ref{minmax}, one has $r_d+r_{d+1} \geq \frac{2m}{d+1}$, 
while $r_d \leq 1$. Hence, $$M \geq r_d+r_{d+1} \geq \frac{2m}{d+1} 
> \frac{2}{d+1}\left( \left\lfloor \frac{d+1}{2} \right\rfloor +1 \right) 
\geq \frac{2}{d+1}\left( \frac{d}{2} +1 \right)=1+\frac{1}{d+1}.$$ 

{\bf The second step.} 
Assume that $m-1 \leq \lfloor \frac{d+1}{2} \rfloor$ and 
let $d+1=km+q$, where $k$ is a positive integer and $0 \leq q \leq m-1$, 
i.e., $k$ (resp. $q$) is a quotient (resp. a remainder) of $d+1$ divided by $m$. 

If we suppose that $\sum_{j=0}^{k-1}r_{jm+q+1} > 1$, then one has 
\begin{eqnarray*}
1<\sum_{j=0}^{k-1}r_{jm+q+1} \leq \sum_{j=0}^{k-1}r_{jm+q+2} \leq 
\cdots \leq \sum_{j=0}^{k-1}r_{jm+q+m}. 
\end{eqnarray*}
Thus, $m=\sum_{i=1}^{d+1}r_i \geq \sum_{i=q+1}^{d+1}r_i > m$, a contradiction. 
Hence, we have $$\sum_{j=0}^{k-1}r_{jm+q+1} \leq 1.$$ 

{\bf The third step.} 
If we assume that $q \not= m-1$, that is, $0 \leq q \leq m-2$, 
then one has $\sum_{j=0}^{k-2}r_{jm+q+2} \leq \frac{d-q-m+1}{d-q}.$ 
In fact, on the contrary, suppose that 
$\sum_{j=0}^{k-2}r_{jm+q+2} > \frac{d-q-m+1}{d-q}.$ Then, 
\begin{eqnarray*}
\frac{d-q-m+1}{d-q} < \sum_{j=0}^{k-2}r_{jm+q+2} \leq \sum_{j=0}^{k-2}r_{jm+q+3} \leq 
\cdots \leq \sum_{j=0}^{k-2}r_{jm+q+m+1}. 
\end{eqnarray*}
Thus, $\sum_{i=q+2}^{(k-1)m+q+1}r_i>\frac{m(d-q-m+1)}{d-q}$. 
Moreover, since $\sum_{i=q+2}^{d+1}r_i=m-\sum_{i=1}^{q+1}r_i$, 
we also have $\sum_{i=(k-1)m+q+2}^{d+1}r_i \geq \frac{(m-1)(m-\sum_{i=1}^{q+1}r_i)}{d-q}$ 
by Proposition \ref{minmax}. Hence, 
\begin{align*}
m-\sum_{i=1}^{q+1}r_i=\sum_{i=q+2}^{d+1}r_i &> \frac{m(d-q-m+1)}{d-q}+\frac{(m-1)(m-\sum_{i=1}^{q+1}r_i)}{d-q} \\
&= \frac{m(d-q)}{d-q}-\frac{(m-1)\sum_{i=1}^{q+1}r_i}{d-q} \geq m - \sum_{i=1}^{q+1}r_i, 
\end{align*}
a contradiction. Here, since $m-1 \leq \lfloor \frac{d+1}{2} \rfloor \leq \frac{d+1}{2}$ 
and $0 \leq q \leq m-2<d$, we have $m+q \leq 2m-2 \leq d+1$, 
which means that $\frac{m-1}{d-q} \leq 1$. Thus, one has 
$$\sum_{j=0}^{k-2}r_{jm+q+2} \leq \frac{d-q-m+1}{d-q}.$$ 

Similarly, if we assume that $q=m-1$, then one has 
$$\sum_{j=0}^{k-1}r_{jm+1} \leq \frac{d-m+2}{d+1}.$$ 

{\bf The fourth step.} 
In this step, we prove that $$\sum_{j=0}^{k-1}r_{jm+q+1}+r_{d+1} \geq 1+\frac{1}{d+1}.$$ 

We assume that $0 \leq q \leq m-2$. Suppose, on the contrary, 
$\sum_{j=0}^{k-1}r_{jm+q+1}+r_{d+1} < 1+\frac{1}{d+1}$. Then 
$\sum_{j=1}^{k-1}r_{jm+q+1}+r_{d+1} < 1+\frac{1}{d+1}-r_{q+1} < 1+\frac{1}{d-q}-r_{q+1}$. Thus, 
\begin{align*}
1+\frac{1}{d-q}-r_{q+1} &> \sum_{j=1}^{k-1}r_{jm+q+1}+r_{km+q} \geq \sum_{j=1}^{k-1}r_{jm+q}+r_{km+q-1} \geq \cdots \\
&\geq \sum_{j=1}^{k-1}r_{jm+q+1-(m-2)}+r_{km+q-(m-2)}=\sum_{j=0}^{k-2}r_{jm+q+3}+r_{(k-1)m+q+2}. 
\end{align*}
Moreover, by the third step, we also have $\sum_{j=0}^{k-2}r_{jm+q+2} \leq \frac{d-q-m+1}{d-q}$. Hence, 
\begin{align*}
m-\sum_{i=1}^{q+1}r_i=\sum_{i=q+2}^{d+1}r_i &< m-1+\frac{m-1}{d-q}-(m-1)r_{q+1}+\frac{d-q-m+1}{d-q} \\
&= m-(m-1)r_{q+1} \leq m-(q+1)r_{q+1} \leq m-\sum_{i=1}^{q+1}r_i, 
\end{align*}
a contradiction. Similarly, when $q=m-1$, 
if we suppose that $\sum_{j=1}^kr_{jm}+r_{km+m-1} < 1+\frac{1}{d+1}$, then 
\begin{eqnarray*}
1+\frac{1}{d+1}>\sum_{j=1}^kr_{jm}+r_{km+m-1} \geq \sum_{j=1}^kr_{jm-1}+r_{km+m-2} 
\geq \cdots \geq \sum_{j=0}^{k-1}r_{jm+2}+r_{km+1} 
\end{eqnarray*}
and $\sum_{j=0}^{k-1}r_{jm+1} \leq \frac{d-m+2}{d+1}$ by the third step, 
so we obtain $m=\sum_{i=1}^{d+1}r_i<m-1+\frac{m-1}{d+1}+\frac{d-m+2}{d+1}=m$, a contradiction. 

{\bf The fifth step.} 
Thanks to the second and fourth steps, we have 
$$M \geq \sum_{j=0}^{k-1}r_{jm+q+1}+r_{d+1} \geq 1+\frac{1}{d+1},$$ 
as desired. 
\end{proof}

\bigskip

We also prepare another 
\begin{Lemma}\label{Z}
Let $l$ be an integer with $l \geq 2$ and $i_1,\ldots,i_l$ distinct integers. We set 
\begin{eqnarray*}
Z_l(j)=
\frac{\prod_{k=1}^{j-1}\Delta_{i_ki_j}}{\prod_{1 \leq k \leq l, k\not=j}|\Delta_{i_ki_j}|}p_j+ 
\frac{\prod_{k=1}^{j-1}\Delta_{i_ki_{j+1}}}{\prod_{1 \leq k \leq l, k\not=j+1}|\Delta_{i_ki_{j+1}}|}p_{j+1}+
\cdots+\frac{\prod_{k=1}^{j-1}\Delta_{i_ki_l}}{\prod_{1 \leq k \leq l, k\not=l}|\Delta_{i_ki_l}|}p_l 
\end{eqnarray*}
for $2 \leq j \leq l$. Then, for any $2 \leq j \leq l-1$, we have 
\begin{eqnarray*}
&&Z_l(j)=\frac{\prod_{k=1}^{j-1}\Delta_{i_ki_j}}{\prod_{1 \leq k \leq l, k\not=j}|\Delta_{i_ki_j}|}p_j+ 
\frac{1}{\Delta_{i_ji_{j+1}}}Z_l(j+1)-\frac{1}{\Delta_{i_ji_{j+1}}\Delta_{i_ji_{j+2}}}Z_l(j+2)+ \\
&&\quad\quad\quad\quad\quad\quad\quad\quad\quad\quad\quad\quad\quad\quad\quad\quad\quad\quad\quad
\cdots+(-1)^{l-j+1}\frac{1}{\prod_{k=j+1}^l\Delta_{i_ji_k}}Z_l(l). 
\end{eqnarray*}
\end{Lemma}

A proof is given by elementary computations. 

\smallskip

Now, Lemma \ref{Z} says that if $Z_l(j+1),\ldots,Z_l(l)$ are integers, 
then there exists an integer $p_j$ such that $Z_l(j)$ becomes an integer. 
In fact, since 
$$ \frac{1}{\Delta_{i_ji_{j+1}}}Z_l(j+1)-
\cdots+(-1)^{l-j+1}\frac{1}{\prod_{k=j+1}^l\Delta_{i_ji_k}}Z_l(l) 
=\frac{P}{C},$$ 
where $P$ is some integer and $C=\prod_{k=j+1}^l|\Delta_{i_ji_k}|$, and the numerator (resp. the denominator) of 
$\frac{\prod_{k=1}^{j-1}\Delta_{i_ki_j}}{\prod_{1 \leq k \leq l, k\not=j}|\Delta_{i_ki_j}|}$ 
is either 1 or $-1$ (resp. $C$), it is obvious that 
there exists an integer $p_j$ such that $Z_l(j)$ becomes an integer.

\bigskip

Let $\Qc \subset \RR^N$ be an integral convex polytope of dimension $d$. 
In general, when $\ZZ \Ac_{\Qc}=\ZZ^{N+1}$, 
in order to prove that $\Qc$ is normal, it suffices to show that 
for any $\alpha=(m,\alpha_1,\ldots,\alpha_N) \in \ZZ \Ac_{\Qc} \cap \QQ_{\geq 0} \Ac_{\Qc}
=\QQ_{\geq 0} \Ac_{\Qc} \cap \ZZ^{N+1}$ with $m \geq 2$, 
we find $\alpha' \in \Qc^* \cap \ZZ^{N+1}$ and $\alpha'' \in \QQ_{\geq 0} \Ac_{\Qc} \cap \ZZ^{N+1}$ 
with $\alpha=\alpha'+\alpha''$. (This is equivalent to prove that $\Qc$ satisfies the integer decomposition property.) 
In particular, when $\Qc$ is a simplex, since there exists a unique $(r_1,\ldots,r_{d+1}) \in \QQ^{d+1}$ 
such that $\alpha=\sum_{i=1}^{d+1}r_iu_i$ and $\sum_{i=1}^{d+1}r_i=m$, 
where $u_1,\ldots,u_{d+1}$ are the vertices of $\Qc^*$, 
we may find $(r_1',\ldots,r_{d+1}') \in \QQ^{d+1}$ with $\sum_{i=1}^{d+1}r_i'u_i \in \Qc^* \cap \ZZ^{N+1}$ 
and $(r_1'',\ldots,r_{d+1}'') \in \QQ^{d+1}$ with 
$\sum_{i=1}^{d+1}r_i''u_i \in \QQ_{\geq 0} \Ac_{\Qc} \cap \ZZ^{N+1}$ 
satisying $r_i'+r_i''=r_i$ for $1 \leq i \leq d+1$. 

Hence, it is enough to show that for any $\alpha=\sum_{i=1}^{d+1}r_iu_i \in \QQ_{\geq 0} \Ac_{\Qc} \cap \ZZ^{N+1}$ 
with $\sum_{i=1}^{d+1}r_i \geq 2$, 
there exists $(r_1',\ldots,r_{d+1}') \in \QQ^{d+1}$ such that 
$$\sum_{i=1}^{d+1}r_i'=1, \;\; 0 \leq r_i' \leq r_i \; \text{ for } 1 \leq i \leq d+1 \;\; 
\text{ and } \;\; \sum_{i=1}^{d+1}r_i'u_i \in \ZZ^{N+1}.$$

\bigskip

Now, we come to the position to verify the normality of integral cyclic polytopes 
in the case where $n=d+1$ and $\Delta_{i,i+1} \geq d^2-1$ for $1 \leq i \leq d$. 
Let $\Pc$ be such cyclic polytope. 
Let $m$ be an integer with $m \geq 2$ and $\alpha$ an element in 
$\ZZ \Ac_{\Pc} \cap \QQ_{\geq 0} \Ac_{\Pc}=\QQ_{\geq 0} \Ac_{\Pc} \cap \ZZ^{d+1}$ 
with the first coordinate $m$. 
Since $\Pc^*$ is a simplex of dimension $d$, there exists a unique $(r_1,\ldots,r_{d+1}) \in \QQ^{d+1}$, 
where $\sum_{i=1}^{d+1}r_i=m$, such that $\alpha = \sum_{i=1}^{d+1}r_iv_i$. 
Then what we must do is to show that there exists $(r_1',\ldots,r_{d+1}') \in \QQ^{d+1}$ such that 
\begin{eqnarray}\label{assert}
\sum_{i=1}^{d+1}r_i'=1, \;\; 0 \leq r_i' \leq r_i \; \text{ for } 1 \leq i \leq d+1 \;\; 
\text{ and } \;\; \sum_{i=1}^{d+1}r_i'v_i \in \ZZ^{d+1}. 
\end{eqnarray}

{\bf The first step.} 
If there exists $r_i$ with $r_i \geq 1$, say, $r_1$, 
then we may set $r_1'=1$ and $r_2'=\cdots=r_{d+1}'=0$. 
Moreover, when $m \geq d+1$, since $\sum_{i=1}^{d+1}r_i=m$ and $r_i \geq 0$, 
there is at least one $r_i$ with $r_i \geq 1$. 
Thus, we may assume that 
$$2 \leq m \leq d \;\; \text{ and } \;\; 0 \leq r_i \leq 1 
\;\; \text{ for } \;\; 1 \leq i \leq d+1.$$

{\bf The second step.} 
By Lemma \ref{key}, there exist $r_{i_1},\ldots,r_{i_l}$ among $(r_1,\ldots,r_{d+1})$ 
such that $\sum_{j=1}^lr_{i_j} \geq 1 + \frac{1}{d+1}$ and $\sum_{j=1}^{l-1}r_{i_j} \leq 1$, 
where $0 \leq r_{i_1} \leq \cdots \leq r_{i_l} \leq 1$ and $2 \leq l \leq d$, 
although we do not know whether $1 \leq i_1 < \cdots < i_l \leq d+1$. 
Let $r_{i_1},\ldots,r_{i_l}$ be such ones. However, 
we assume that $0 \leq r_{i_l} \leq r_{i_{l-1}} \leq \cdots \leq r_{i_1} \leq 1$, i.e., we have 
$$\sum_{j=2}^l r_{i_j} \leq 1 \;\; \text{ and } \;\; \sum_{j=1}^lr_{i_j} \geq 1+\frac{1}{d+1}.$$ 

Let $D=d^2-1$. Thus, $|\Delta_{ij}| \geq D$ for any $1 \leq i \not= j \leq d+1$. 
Now, we set $\epsilon(l)=\frac{l-1}{D}$ for $2 \leq l \leq d$. 
Then it is easy to see that $\epsilon(l)$ enjoys the following properties: 
\begin{eqnarray}\label{property}
&&\epsilon(l) \geq \sum_{a=2}^l\frac{1}{D^{a-1}}, \;\;\;\; 
\frac{1}{d+1} = \epsilon(d) > \epsilon(d-1) > \cdots > \epsilon(2), \\
\nonumber &&\epsilon(l)-\frac{l-j+1}{D^{j-1}} > \epsilon(j-1) \; \text{ for } \; 3 \leq j \leq l. 
\end{eqnarray}

In the following two steps, by induction on $l$, we prove that 
if $\sum_{j=1}^lr_{i_j} \geq 1+\epsilon(l)$ and $\sum_{j=2}^lr_{i_j} \leq 1$, 
then there is $(r_1',\ldots,r_{d+1}') \in \QQ^{d+1}$ which satisfies \eqref{assert}. 
Once we know this, we obtain the required assertion 
from $2 \leq l \leq d$ and $\frac{1}{d+1} = \epsilon(d) \geq \epsilon(l)$. 

{\bf The third step.} 
Assume that $l=2$, i.e., we have $r_{i_1}+r_{i_2} \geq 1+\frac{1}{D}$, 
where $0 \leq r_{i_2} \leq r_{i_1} \leq 1$. 

Let $p$ be a nonnegative integer satisfying 
$$\frac{p}{|\Delta_{i_1i_2}|} \leq r_{i_2} < \frac{p+1}{|\Delta_{i_1i_2}|}.$$ 
Then it is clear that there exists such a unique nonnegative integer $p$. 
Let $r_{i_2}'=\frac{p}{|\Delta_{i_1i_2}|}, r_{i_1}'=1-r_{i_2}'$ and 
$r_j'=0$ for any $j$ with $j \in [d+1] \setminus \{i_1,i_2\}$. 
Thus, $\sum_{i=1}^{d+1}r_i'=1$ and $0 \leq r_{i_2}' \leq r_{i_2}$. 
Moreover, since $r_{i_2} \leq 1$, we have $r_{i_1}' = 1 - r_{i_2}' \geq 1 - r_{i_2} \geq 0$. 
In addition, by $r_{i_1}+r_{i_2} \geq 1+\frac{1}{D}$ and $|\Delta_{i_1i_2}| \geq D$, we also have 
\begin{align*}
r_{i_1}-r_{i_1}'=r_{i_1}-1+\frac{p}{|\Delta_{i_1i_2}|} 
\geq \frac{1}{D}-r_{i_2}+\frac{p}{|\Delta_{i_1i_2}|} \geq \frac{p+1}{|\Delta_{i_1i_2}|}-r_{i_2} > 0. 
\end{align*}
On the other hand, by Proposition \ref{lemma:delta}, we may consider $v_{i_1}$ and $v_{i_2}$ as 
$v_{i_1}=(1,0,\ldots,0) \;\; \text{ and } \;\; v_{i_2}=(1,\Delta_{i_1i_2},0,\ldots,0)$. 
Obviously, $\sum_{i=1}^{d+1}r_i'v_i \in \ZZ^{d+1}$. 

{\bf The fourth step.} 
Assume that $l \geq 3$. 
For each $j$ with $2 \leq j \leq l$, we define each nonnegative integer $p_j$ as follows. 
Let $p_l$ be a nonnegative integer which satisfies 
$$\frac{p_l}{\prod_{k=1}^{l-1}|\Delta_{i_ki_l}|} 
\leq r_{i_l} < \frac{p_l+1}{\prod_{k=1}^{l-1}|\Delta_{i_ki_l}|},$$ 
and for $2 \leq j \leq l-1$, let $p_j$ be an integer which satisfies 
$Z_l(j) \in \ZZ$ and 
$$\frac{p_j}{\prod_{1 \leq k \leq l, k \not= j}|\Delta_{i_ki_j}|} 
\leq r_{i_j} < \frac{p_j+\prod_{k=j+1}^l|\Delta_{i_ji_k}|}{\prod_{1 \leq k \leq l, k \not= j}|\Delta_{i_ki_j}|},$$ 
where $Z_l(j)$ is as in Lemma \ref{Z}. 
Thanks to Lemma \ref{Z}, if $Z_l(j+1),\ldots,Z_l(l) \in \ZZ$, 
then there exists an integer $p_j$ with $Z_l(j) \in \ZZ$ and each $p_j$ is uniquely determined 
by the above inequalities. 
Remark that we do not know whether $p_j$ is nonnegative except for $p_l$. 
However, in our case, we may assume that $p_2,\ldots,p_{l-1}$ are all nonnegative because of the following discussions. 
In fact, on the contrary, suppose that there is $j'$ with $p_{j'} < 0$. 
Let $q_{j'} \in \ZZ_{\geq 0}$ be a minimal nonnegative integer satisfying 
\begin{eqnarray*}
&&\frac{\prod_{k=1}^{j'-1}\Delta_{i_ki_{j'}}}{\prod_{1 \leq k \leq l, k\not=j'}|\Delta_{i_ki_{j'}}|}q_{j'}+ 
\frac{1}{\Delta_{i_{j'}i_{j'+1}}}Z_l(j'+1)-\frac{1}{\Delta_{i_{j'}i_{j'+1}}\Delta_{i_{j'}i_{j'+2}}}Z_l(j'+2)+ \\
&&\quad\quad\quad\quad\quad\quad\quad\quad\quad\quad\quad\quad\quad\quad\quad\quad\quad
\cdots+(-1)^{l-j'+1}\frac{1}{\prod_{k=j'+1}^l\Delta_{i_{j'}i_k}}Z_l(l) \in \ZZ. 
\end{eqnarray*}
In particular, it follows from the minimality of $q_{j'}$ that $0 \leq q_{j'} < \prod_{k=j'+1}^l|\Delta_{i_{j'}i_k}|$. 
By our assumption, one has 
$\frac{q_{j'}}{\prod_{1 \leq k \leq l, k \not= j'}|\Delta_{i_ki_{j'}}|} > r_{i_{j'}}.$ Thus, 
\begin{eqnarray*}
r_{i_l} \leq \cdots \leq r_{i_{j'}} < 
\frac{q_{j'}}{\prod_{1 \leq k \leq l, k \not= j'}|\Delta_{i_{j'}i_k}|} 
< \frac{\prod_{k=j'+1}^l|\Delta_{i_{j'}i_k}|}{\prod_{1 \leq k \leq l, k \not= j'}|\Delta_{i_ki_{j'}}|} 
= \frac{1}{\prod_{k=1}^{j'-1}|\Delta_{i_ki_{j'}}|} \leq \frac{1}{D^{j'-1}}, 
\end{eqnarray*}
so one has $\sum_{j=j'}^l r_{i_j} < \frac{l-j'+1}{D^{j'-1}}$. 
From $\sum_{j=1}^l r_{i_j} \geq 1 + \epsilon(l)$ and \eqref{property}, we have 
$$\sum_{j=1}^{j'-1} r_{i_j} > 1+\epsilon(l)-\frac{l-j'+1}{D^{j'-1}} > 1+\epsilon(j'-1)$$ 
when $j' \geq 3$. Hence, we may skip such case by the hypothesis of induction. 
When $j'=2$, one has $r_{i_1}>1+\epsilon(l)-\frac{l-1}{D}=1$, a contradiction. 

By using the above $p_j$'s, we define $r_1',\ldots,r_{d+1}'$ by setting 
\begin{eqnarray*}
r_a'=
\begin{cases}
\displaystyle \frac{p_j}{\prod_{1 \leq k \leq l, k \not= j}|\Delta_{i_ki_j}|}, 
\;\;\; &\text{ if } a=i_j \in \{i_2,\ldots,i_l\}, \\
\displaystyle 1 - \sum_{j=2}^lr_{i_j}', &\text{ if } a=i_1, \\
0, &\text{ otherwise}. 
\end{cases}
\end{eqnarray*}
In particular, $\sum_{a=1}^{d+1}r_a'=1$. By definition of $r_{i_2}',\ldots,r_{i_l}'$, 
we have $0 \leq r_{i_j}' \leq r_{i_j}$ for $2 \leq j \leq l$. 
Moreover, from $\sum_{j=2}^lr_{i_j} \leq 1$, we also have 
$r_{i_1}'=1-\sum_{j=2}^lr_{i_j}' \geq 1-\sum_{j=2}^lr_{i_j} \geq 0$. 
In addition, from $\sum_{j=1}^lr_{i_j} \geq 1+\epsilon(l)$ and \eqref{property}, we also have 
\begin{align*}
r_{i_1}-r_{i_1}'&=r_{i_1}-1+\sum_{j=2}^l\frac{p_j}{\prod_{1 \leq k \leq l, k \not=j}|\Delta_{i_ki_j}|} 
\geq \epsilon(l)-\sum_{j=2}^lr_{i_j}+\sum_{j=2}^l\frac{p_j}{\prod_{1 \leq k \leq l, k \not=j}|\Delta_{i_ki_j}|} \\
&\geq \sum_{j=2}^l\left( \frac{p_j}{\prod_{1 \leq k \leq l, k \not=j}|\Delta_{i_ki_j}|}+\frac{1}{D^{j-1}}-r_{i_j} \right) 
\geq \sum_{j=2}^l
\left( \frac{p_j+\prod_{k=j+1}^l|\Delta_{i_ji_k}|}{\prod_{1 \leq k \leq l, k \not= j}|\Delta_{i_ki_j}|} -r_{i_j} \right) \\
&> 0. 
\end{align*}
Finally, we verify that $\sum_{i=1}^{d+1}r_i'v_i \in \ZZ^{d+1}$. 
Again, by Proposition \ref{lemma:delta}, we may consider $v_{i_1},\ldots,v_{i_l}$ as follows: 
\begin{eqnarray*}
\begin{pmatrix}
v_{i_1} \\
v_{i_2} \\
\vdots \\
v_{i_l} 
\end{pmatrix}=
\begin{pmatrix}
1      &0               &\cdots                         &\cdots &\cdots                           &0      &\cdots &0 \\
1      &\Delta_{i_1i_2} &0                              &\ddots &\ddots                           &\vdots &       &\vdots \\
1      &\Delta_{i_1i_3} &\Delta_{i_1i_3}\Delta_{i_2i_3} &\ddots &\ddots                           &\vdots &       &\vdots \\
\vdots &\vdots          &\vdots                         &\ddots &\ddots                           &\vdots &       &\vdots \\
1      &\Delta_{i_1i_l} &\Delta_{i_1i_l}\Delta_{i_2i_l} &\cdots &\prod_{k=1}^{l-1}\Delta_{i_ki_l} &0      &\cdots &0 
\end{pmatrix}.
\end{eqnarray*}
Hence, it is easy to check that 
$$\sum_{i=1}^{d+1}r_i'v_i=\sum_{j=1}^lr_{i_j}'v_{i_j}=(1,Z_l(2),Z_l(3),\ldots,Z_l(l),0,\ldots,0) \in \ZZ^{d+1},$$ 
proving the assertion.

\begin{Remark}{\em 
Since each lattice length of an edge $\con(\set{v_i,v_j})$ of $\Pc^*$ coincides with 
$\Delta_{ij}$, where $i<j$, it follows immediately from \cite[Theorem 1.3 (b)]{Gubel}
that $\Pc$ is normal if $\Delta_{i,i+1} \geq d(d+1)$ for $1 \leq i \leq n-1$. 
(We are grateful to G\'abor Heged\"us for informing us the result \cite[Theorem 1.3 (b)]{Gubel}.) 
Thus, our constraint $\Delta_{i,i+1} \geq d^2-1$ on integral cyclic polytopes 
is better than a general case, but this bound is still very rough. 
For example, $C_3(0,1,2,3)$ is normal, 
while we have $\Delta_{12}=\Delta_{23}=\Delta_{34}=1 < 8$. 
Similarly, $C_4(0,1,3,5,6)$ is also normal, 
although one has $\Delta_{12}=\Delta_{45}=1$ and $\Delta_{23}=\Delta_{34}=2$. 
}\end{Remark}


\section{Cyclic polytopes that are not very ample}

Our goal of this section is to prove 
\begin{Theorem}\label{main2}
Let $d$ and $n$ be positive integers satisfying $n \geq d+1$ and $d \geq 4$. 
If $\Delta_{12}=1$ or $\Delta_{n-2,n-1}=1$, then $C_d(\tau_1,\ldots,\tau_n)$ is not very ample. 
\end{Theorem}
We obtain Theorem \ref{main2} as a conclusion of 
Proposition \ref{non-normal4} and Corollary \ref{non-normal5} below. 
\begin{Proposition}\label{non-normal4}
Let $\Pc = C_4(\tau_1,\dotsc, \tau_n)$. 
If $\Delta_{23} = 1$ or $\Delta_{n-2,n-1} = 1$, then $\Pc$ is not very ample. 
\end{Proposition}
\begin{proof}
Thanks to Lemma \ref{equiv}, by symmetry, we assume $\Delta_{23} = 1$. Consider the set
\[\Ac_{\Pc,3} \defa \set{x-v_3 : x \in \Pc^* \cap \ZZ^5}. \] 
We will prove that the monoid $\ZZ_{\geq 0} \Ac_{\Pc,3}$ is not normal, 
thus there exists a vector $p \in \ZZ \Ac_{\Pc, 3} \cap \QQ_{\geq 0} \Ac_{\Pc,3}=\QQ_{\geq 0} \Ac_{\Pc, 3} \cap \ZZ^5$ 
such that $p \notin \ZZ_{\geq 0} \Ac_{\Pc,3}$. 
Then, for every integer $k \geq 1$, it holds that 
$k v_3 + p \in (\ZZ\Ac_\Pc \cap \QQ_{\geq 0}\Ac_\Pc) \setminus \ZZ_{\geq 0} \Ac_\Pc$, 
see \cite[Excercise 2.23]{brunsgubel}. Hence, $\Pc$ is not very ample.
 
In the sequel, we denote the facet of $\Pc^*$ spanned by the vertices $v_i,v_j, v_k$ and $v_l$ with $\Fc_{ijkl}$. 
Moreover, we denote the corresponding linear form with $\sigma_{ijkl}$. 
Note that every facet of $\Pc^*$ containing $v_3$ defines also a facet of $\QQ_{\geq 0} \Ac_{\Pc,3}$. 

The following vector has the required properties:
\begin{align*}
p &\defa \bvec{23} + \bvec{134}+\bvec{12345} \\
  &= \frac{\Delta_{12}\Delta_{15}+1}{\Delta_{12}\Delta_{13}\Delta_{14}\Delta_{15}} v_1 + 
     \frac{1}{\Delta_{23}}\left(1-\frac{1}{\Delta_{12}\Delta_{24}\Delta_{25}}\right) v_2 -
     \frac{1}{\Delta_{23}}\left(1+\frac{\Delta_{23}\Delta_{35}-1}{\Delta_{13}\Delta_{34}\Delta_{35}}\right) v_3 \\ 
  &+ \frac{\Delta_{24}\Delta_{45}-1}{\Delta_{14}\Delta_{24}\Delta_{34}\Delta_{45}} v_4 + 
     \frac{1}{\Delta_{15}\Delta_{25}\Delta_{35}\Delta_{45}} v_5. 
\end{align*}
First, one has $p \in \ZZ^5$ from Proposition \ref{prop:bvectors} (1). 
Then, by the second representation of $p$, it is a positive linear combination of the vectors 
$v_1-v_3, v_2-v_3, v_4-v_3$ and $v_5-v_3$. Thus, $p \in \QQ_{\geq 0} \Ac_{\Pc,3}$. 
Moreover, since we assume $\Delta_{23} = 1$, 
the coefficient of $v_3$ is less than $-1$. Hence, $p$ lies beyond the facet $\Fc_{1245}$ 
which is a facet of $\Pc^*$ by Gale's evenness condition (Proposition \ref{galesevenness}). 
Thus, we have $p \notin \Ac_{\Pc,3}$. 

It remains to show that $p$ cannot be written as a sum $\sum w_j$ with $w_j \in \Ac_{\Pc,3}$. 
Suppose that we have such a representation. Then we remark that $p$ has at least two summands. 
Consider a facet $\Fc_{1234}$. 
Then $\sigma_{1234}(p) = \frac{1}{\Delta_{15}\Delta_{25}\Delta_{35}\Delta_{45}} \sigma_{1234}(v_5) = 1$. 
Since $\sigma_{1234}(w_j) \geq 0$, $\sigma_{1234}(w_j) = 0$ for every summand $w_j$ except one. 
Choose one $w_j \neq 0$ with $\sigma_{1234}(w_j) = 0$ and denote it by $w$. 
Further, we set $w' \defa p - w \in \ZZ_{\geq 0} \Ac_{\Pc,3}$ the remaining sum. 
By Carath\'{e}odory's Theorem, 
there exist vertices $v_{i_1},\dotsc,v_{i_4}$ of $\Pc^*$ and nonnegative numbers $\lambda_j \geq 0$, 
such that $w' = \sum_{j=1}^4 \lambda_j (v_{i_j} - v_3)$. 
Let $i_4$ be the greatest of those indices. Since $\sigma_{1234}(w') = 1$ 
and $\sigma_{1234}(v_{i_4}) = \Delta_{1i_4}\Delta_{2i_4}\Delta_{3i_4}\Delta_{4i_4}$, 
we conclude that \[ \lambda_4 \leq \frac{1}{\Delta_{1i_4}\Delta_{2i_4}\Delta_{3i_4}\Delta_{4i_4}} \,.\]
But the vertices $v_{i_1},\dotsc,v_{i_4}$ and $v_3$ define an integral cyclic polytope, 
thus the denominator of the coefficient of $v_{i_4}$ has to be a divisor of 
$\Delta_{i_1 i_4}\Delta_{i_2 i_4}\Delta_{i_3 i_4}\Delta_{3 i_4}$. 
This is only possible if $\set{i_1,i_2,i_3} = \set{1,2,4}$. 
Thus, $w'$ lies in the cone generated by $v_1 - v_3, v_2 - v_3, v_4 - v_3$ and $v_{i_4} - v_3$.
Note that $\sigma_{1234}(w) = 0$ implies that $w$ lies in the cone 
generated by $v_1 - v_3, v_2 - v_3$ and $v_4 - v_3$. 
Thus we can replace the polytope $\Pc^*$ by the polytope $\Qc^*$ 
whose vertices are $v_1,\dotsc,v_5$ and $v_{i_4}$. 
The reason for doing this is that we know the facets of $\Qc^*$. 
Here, $i_4 = 5$ is possible.

We consider the representation 
\[ w = a_1 \bvec{3} + a_2 \bvec{23}+a_3 \bvec{123} + a_4 \bvec{1234} \]
with integer coefficients $a_1, a_2, a_3, a_4$. 
This is possible from Proposition \ref{prop:bvectors} (3).
Since $w$ is in the cone generated by $v_1 - v_3, v_2 - v_3$ and $v_4 - v_3$, 
we have $a_1=0$. 
Now consider a facet $\Fc_{123i_4}$ of $\Qc^*$. We compute 
\[ \sigma_{123i_4}(p) = \frac{1}{\Delta_{45}}(\Delta_{24}\Delta_{45}-1)\Delta_{4i_4} + \frac{1}{\Delta_{45}} \Delta_{5i_4} 
= \Delta_{24}\Delta_{4i_4}-1. \]
Moreover, $\sigma_{123i_4}(w) = - a_4 \Delta_{4i_4}$. From $0 \leq \sigma_{123i_4}(w) \leq \sigma_{123i_4}(p)$, 
we conclude $0 \leq - a_4 \leq \Delta_{24} - 1$. Here we used that $a_4$ is an integer. 
Next, consider a facet $\Fc_{2345}$. We compute 
$\sigma_{2345}(w) = a_3 \Delta_{14}\Delta_{15} + a_4 \Delta_{15}$ and $\sigma_{2345}(p) = \Delta_{12}\Delta_{15}+1$. 
As before, we conclude that $0 \leq a_3 \Delta_{14} + a_4 \leq \Delta_{12}$. 
However, these two constraints can only be satisfied by $a_3 = a_4 = 0$, 
because $\Delta_{14} = \Delta_{12} + \Delta_{24}$ and $\Delta_{15}>1$. 
Finally, we consider a facet $\Fc_{134i_4}$. 
By computing $\sigma_{134i_4}(w) = a_2 \Delta_{12}\Delta_{24}\Delta_{2i_4}$ and 
$\sigma_{134i_4}(p) = \Delta_{12}\Delta_{24}\Delta_{2i_4} - 1$, 
we conclude that $a_2 = 0$. But this means $w=0$, a contradiction to $w \neq 0$.
\end{proof}

By using this proposition, we also obtain 
\begin{Corollary}\label{non-normal5}
Let $\Pc=C_d(\tau_1,\ldots,\tau_n)$, where $d \geq 5$. 
If there is some $i$ with $2 \leq i \leq n-2$ such that $\Delta_{i,i+1} = 1$, 
then $\Pc$ is not very ample. 
\end{Corollary}
\begin{proof}
We prove this by induction on $d$. 

When $d=5$, let $\Fc_i=\con(\{v_1,v_i,v_{i+1},v_{i+2},v_{i+3}\})$ for $2 \leq i \leq n-3$ 
and $\Fc_{n-2}=\con(\{v_{n-4},v_{n-3},v_{n-2},v_{n-1},v_n\})$. 
By Gale's evenness condition, each $\Fc_i$ is a facet of $\Pc^*$. 
When $\Delta_{i,i+1} = 1$ for some $i$ with $2 \leq i \leq n-2$, 
it then follows from Proposition \ref{non-normal4} that $\Fc_i$ is not very ample. 
Thus, $\Pc$ itself is not very ample, either. (See \cite[Lemma 1]{OH}.) 

Now, let $d \geq 6$. For $2 \leq i \leq n-d+2$, we set 
\begin{eqnarray*}
\Fc_i=
\begin{cases}
\con(\{v_1,v_i,\ldots,v_{i+d-2}\}), \;\;\;\; &\text{when} \; d \; \text{is odd}, \\
\con(\{v_{i-1},v_i,\ldots,v_{i+d-2}\}), &\text{when} \; d \; \text{is even}. 
\end{cases}
\end{eqnarray*}
Again, Gale's evenness condition guarantees that each $\Fc_i$ is a facet of $\Pc^*$. 
When $\Delta_{i,i+1}=1$ for some $i$ with $2 \leq i \leq n-2$, 
since each facet is also an integral cyclic polytope of dimension $d-1$, 
either $\Fc_i$ or $\Fc_{d-n+2}$ is not very ample by the hypothesis of induction. 
Therefore, $\Pc$ is not very ample. 
\end{proof}

On the case where $d = 2$, 
it is well known that there exists a unimodular triangulation for 
every integral convex polytope of dimension 2. 
Therefore, integral convex polytopes of dimension 2 are always normal. 

On the case where $d = 3$, 
exhaustive computational experiences lead us to give the following 
\begin{Conjecture}\label{conj:3}
All cyclic polytopes of dimension $3$ are normal. 
\end{Conjecture}

Moreover, by computational experiences together with Proposition \ref{non-normal4}, 
we also conjecture a complete characterization of normal cyclic polytopes of dimension 4. 
\begin{Conjecture}\label{conj:4}
A cyclic polytope of dimension $4$ is normal if and only if we have 
$$\Delta_{23} \geq 2 \;\;\; \text{and} \;\;\; \Delta_{n-2,n-1} \geq 2.$$ 
\end{Conjecture}

By considering the foregoing two conjectures and Theorem \ref{main1}, 
the following statement seems to be natural for us. 
\begin{Conjecture}
If $\Pc = C_d(\tau_1,\dotsc,\tau_n)$ is normal and $\Pc' = C_d(\tau_1',\dotsc,\tau_n')$ 
satisfies $\tau_j'-\tau_i' \geq \Delta_{ij}$ for all $1\leq i < j \leq n$, then $\Pc'$ is also normal.
\end{Conjecture}

Finally, we also state 
\begin{Conjecture}
If an integral cyclic polytope is very ample, then it is also normal. 
\end{Conjecture}

Actually, it often happens that a very ample integral convex polytope 
is also normal, that is to say, the normality of an integral convex polytope 
is equivalent to what it is very ample. 
Hence, the above conjecture occurs in the natural way. 
On the other hand, it is also known that there exists an integral convex polytope 
which is not normal but very ample. See \cite[Exercise 2.24]{brunsgubel}.


\end{document}